\documentclass[letter, 9pt, twocolumn,conference]{ieeeconf}
\IEEEoverridecommandlockouts                              
\usepackage{amsmath,amssymb}
\usepackage{amsfonts}
\usepackage{mathrsfs}
\usepackage{color}
\usepackage[dvips]{graphicx}
\usepackage{epsfig,graphicx,pstricks}

\newtheorem{proposition}{Proposition}

\newtheorem{remark}{Remark}

\def\begcen{\begin{center}}
\def\endcen{\end{center}}

\newcommand{\col}{ \mbox{col} }

\def\cala{{\cal A}}

\def\hal{{1 \over 2}}

\def\L1ac{${\cal L}_1$--AC}
\def\L1{{\cal L}_1}
\def\Linf{{\cal L}_\infty}
\def\L2{{\cal L}_2}
\def\L2e{{\cal L}_{2e}}

\def\rea{\mathbb{R}}

\def\begequarr{\begin{eqnarray}}
\def\endequarr{\end{eqnarray}}
\def\begequarrs{\begin{eqnarray*}}
\def\endequarrs{\end{eqnarray*}}
\def\begarr{\begin{array}}
\def\endarr{\end{array}}
\def\begequ{\begin{equation}}
\def\endequ{\end{equation}}
\def\lab{\label}
\def\begdes{\begin{description}}
\def\enddes{\end{description}}
\def\begenu{\begin{enumerate}}
\def\begite{\begin{itemize}}
\def\endite{\end{itemize}}
\def\endenu{\end{enumerate}}

\def\lef[{\left[\begin{array}}
\def\rig]{\end{array}\right]}

\def\begcen{\begin{center}}
\def\endcen{\end{center}}
\def\begrem{\begin{remark}\rm}
\def\endrem{\end{remark}}

\def\TAC{{\it IEEE Transactions of Automatic Control}}

\def\SCL{{\it Systems and Control Letters}}

\def\IJC{{\it Int. J. of Control}}
\def\CSM{{\it IEEE Control Systems Magazine}}

\begin{document}
\title{Adaptation is Unnecessary in ${\cal L}_1$--``Adaptive" Control}
\author{ {\bf Romeo Ortega and Elena Panteley} \\
         Laboratoire des Signaux et Syst\`emes \\
        CNRS--Supelec\\
        Plateau du Moulon \\
        91192 Gif-sur-Yvette, France\\
        {\tt ortega[panteley]@lss.supelec.fr}
        }

\maketitle
%
\section{What makes ``adaptive" an adaptive controller?}
\lab{sec1}
%
The basic {premise} upon which adaptive control is based is the existence of a parameterized controller that achieves the control objective. It is, moreover, assumed that these parameters are not
known but that they can be estimated on--line from measurements of the plant signals. Towards this end, an identifier is added to generate the parameter estimates. Then, applying in an {\em ad--hoc}
manner a certainty equivalence principle, these estimates are directly applied in the aforementioned control law.

Let us illustrate the discussion above with the simplest example of direct, adaptive, state--feedback stabilization of single--input, linear
time--invariant (LTI) system of the form
\begequ
\lab{pla}
\dot x = Ax + bu
\endequ
where the state $x\in \rea^n$ is assumed to be {\em measurable},  $u\in\rea$ is the control signal, $A \in \rea^{n \times n}$ is the system matrix
and $b \in \rea^n$ the input vector. It is assumed that there exists a vector $\theta \in \rea^n$ such that
$$
A+b\theta^\top=:A_m
$$
is a Hurwitz matrix, but this vector is {\em unknown}. In this case, the ideal control law takes the form
\begequ
\lab{idecon}
u = \theta^\top x,
\endequ
that, as mentioned above, is made adaptive adding an identifier that generates the estimated parameters $\hat \theta \in \rea^n$. In this way, we
obtain the adaptive control law
\begequ
\lab{adacon}
 u = \hat \theta^\top x.
\endequ
Defining the parameter error
\begequ
\lab{tilthe}
\tilde \theta:=\hat \theta - \theta,
\endequ
the control law may be written as
$$
u = \theta^\top x + \tilde \theta^\top x.
$$
If the parameter estimates converge to the desired value $\theta$ the control signal converges to the ideal control law \eqref{idecon} and asymptotic stabilization is achieved---provided $x$ remains
bounded.\footnote{Actually, to achieve stabilization it is enough that $\hat \theta$ converges to the set $\{K \in \rea^n\;|\; A+bK^\top \mbox{ \;is \;Hurwitz\;}\}$. This is the fundamental
self--tuning property of direct adaptive control.}

A key observation is that the ideal control signal  \eqref{idecon} {\em cannot be implemented} without knowledge of the unknown parameters. If this were not the case adaptation would be unnecessary
and we simply would plug in the controller that results when $\tilde \theta=0$!

In a (long) series of recent papers---see, {\em e.g.},  \cite{HOVetal} and the extensive list of references therein---it has been proposed to
replace \eqref{adacon} by
\begequ
\lab{sysu}
\dot u = -k (u-\hat \theta^\top x),
\endequ
where $k>0$ is a design parameter. Combining \eqref{sysu} with a standard state prediction--based estimator is called in \cite{HOVetal} ${\cal
L}_1$--adaptive control, which in the sequel we refer to as ${\cal L}_1$--AC.

The purpose of this paper is to prove, via a proposition given below, that adaptation is {\em unnecessary} in ${\cal L}_1$--AC in the following precise sense.
 {
\begite
\item[F1] For any parameter estimation law, the control signal \eqref{sysu} {\em exactly} coincides with the output of the LTI, full--state feedback, perturbed,  PI controller
\begequarr
\nonumber
\dot v & = & -K_I^\top x + k \tilde \theta^\top x\\
u & = & v -K_P^\top x,
\lab{pi}
\endequarr
with the gains $K_P,K_I \in \rea^n$ are {\em independent} of the parameters $\theta$.
\item[F2] The term $\tilde \theta^\top x$ converges to zero. Hence, the ${\cal L}_1$--AC {\em always} converges to a controller that can be obtained {\em without} knowledge of the unknown parameters.

\item[F3] If the implementable PI controller\footnote{In the sequel we will say that a controller is ``implementable" if its gains are independent of the unknown plant parameters.}
\begequarr
\nonumber
\dot v & = & -K_I^\top x \\
u & = & v -K_P^\top x,
\lab{pi1}
\endequarr
{\em does not stabilize} the plant \eqref{pla} then the ${\cal L}_1$--AC does not stabilize it either.
\endite
}
%
\section{Is adaptation necessary in ${\cal L}_1$--adaptive control?}
\lab{sec2}
%
We analyze in this paper the ${\cal L}_1$--AC proposed in \cite{HOVetal} to address the basic problem of stabilization of single--input, LTI systems discussed in the previous section. In ${\cal L}_1$--AC,
besides the (overly restrictive) assumption of measurable state, it is assumed  that the system can be represented in canonical form
\begequ
\lab{a}
A = \lef[{ccccc} 0 & 1 & 0 & \dots & 0  \\
                  0 & 0 & 1 & \dots & 0  \\
                   \vdots & \vdots & \vdots & \vdots & \vdots  \\
                    0 & 0 & 0 & \dots & 1  \\
                    -a_1 & -a_2 & -a_3 & \dots & -a_n
   \rig],
\endequ
where $a_i \in \rea,\;i\in \bar n:=\{1,\dots,n\}$ are unknown coefficients, and that the input vector $b$ is {\em known}. In the sequel we set $b=e_n$, the $n$--th vector of the Euclidean basis, which is done
without loss of generality in view of the assumption of known $b$. The system \eqref{pla} can also be expressed in the form
\begequ
\lab{sysx}
\dot x  =  A_m x - b (\theta^\top x -  u)
\endequ
with
\begequ
\lab{am}
A_m = \lef[{ccccc} 0 & 1 & 0 & \dots & 0  \\
                  0 & 0 & 1 & \dots & 0  \\
                   \vdots & \vdots & \vdots & \vdots & \vdots  \\
                      0 & 0 & 0 & \dots & 1  \\
                    -a^m_1 & -a^m_2 & -a^m_3 & \dots & -a^m_n
   \rig]
\endequ
where $a^m_i >0,\;i\in \bar n$, are designer chosen coefficients and $\theta \in \rea^n$ is a vector of {\em unknown} parameters, given by
\begequ
\lab{the}
\theta=\col(a_1 -a^m_1, a_2 -a^m_2,\dots, a_n-a^m_n),
\endequ
where $\col(\cdot)$ denotes column vector. In the ${\cal L}_1$--AC proposed in \cite{HOVetal} the control law is computed via \eqref{sysu}. The parameters are updated using the classical state predictor--based
estimator
\begequarr
\nonumber
\dot {\hat x} & = & A_m \hat x - b (\hat \theta^\top x -  u)\\
\dot {\hat \theta} & = &  \gamma x (\hat x-x)^\top Pb
\lab{syshatx}
\endequarr
where $\gamma >0$ is the adaptation gain and $P>0$ is a Lyapunov matrix for $A_m$, that is,
$$
P A_m + A_m^\top P = - Q,
$$
where $Q\in \rea^{n \times n}$ is a positive definite matrix.

{   The proposition below formally establishes the facts F1--F3 stated in the previous section.\\
\begin{proposition}
\lab{pro0}
Consider the plant \eqref{pla} with $A$ given by \eqref{a} and $b=e_n$.
\begenu
\item[P1] Independently of the parameter estimation, the signal $u$ generated by the ${\cal L}_1$--AC control law \eqref{sysu}
exactly coincides with the output of the perturbed, full-state feedback, LTI, {\em implementable} PI controller \eqref{pi} with
\begequarr
\nonumber K_I & = & k \; \col( a^m_1, a^m_2, a^m_3, \dots, a^m_n)\\
K_P & = & k \;e_n.
\lab{pigai}
\endequarr
\item[P2]  If the ${\cal L}_1$--AC controller \eqref{sysu}, \eqref{syshatx} ensures boundedness of trajectories then the perturbation term verifies
\begequ
\lab{lim}
\lim_{t \to \infty}|\tilde \theta^\top(t)x(t)|=0.
\endequ
Consequently, the (bounded state) ${\cal L}_1$--AC {\em always} converges to the PI controller.
\item[P3]  If the PI \eqref{pi1}, \eqref{pigai} {\em does not} ensure stability of the closed--loop system then the ${\cal L}_1$--AC \eqref{sysu}, \eqref{syshatx} does not ensure boundedness of
trajectories.
\endenu
\end{proposition}

\begin{proof}
To establish P1 we use the definition of the parameter error \eqref{tilthe} to write the control signal \eqref{sysu} as
\begequ
\lab{cdotu}
\dot u  = -k(u- \theta^\top x)+ k\tilde \theta^\top x.
\endequ
Now, pre--multiplying the plant dynamics \eqref{sysx}---that is equivalent to  \eqref{pla}---by $e_n^\top$, and rearranging terms, we get
$$
u- \theta^\top x=e_n^\top(\dot x - A_mx),
$$
that, upon replacement in \eqref{cdotu}, yields
$$
\dot u  = -ke_n^\top(\dot x - A_mx)+k \tilde \theta^\top x.
$$
The proof is completed defining the signal
$$
v=u+k e_n^\top x,
$$
and using the definition of $A_m$ given in \eqref{am}.

To prove P2 we first write the dynamics of the system \eqref{sysx} in closed--loop with the ${\cal L}_1$--AC   \eqref{syshatx}, \eqref{cdotu},
\begequarr
\nonumber
\dot {\tilde x} & = &  A_m\tilde x - b \tilde \theta^\top x \\
\nonumber
\dot {\tilde \theta} & = &  \gamma x\tilde x^\top Pb\\
\lef[{c} \dot x \\ \dot u \rig]&=&\cala_0 \lef[{c} x \\ u\rig] + \lef[{c} 0 \\ k\tilde \theta^\top x\rig],
\lab{errequ2}
\endequarr
where\footnote{This matrix is reported in equation (12) of \cite{CAOHOV}.}
$$
\cala_0 = \lef[{cc} A & b  \\ k \theta^\top & -k \rig],
$$
and $\tilde x  =  \hat x - x$ is the prediction error. Consider the function
$$
V(\tilde x, \tilde \theta)=\hal \tilde x^\top P \tilde x +{1 \over 2 \gamma}|\tilde \theta|^2,
$$
whose derivative along the trajectories of \eqref{errequ2} is
$$
\dot V = -\hal \tilde x^\top Q \tilde x.
$$
Since it has been assumed that all trajectories are bounded we can invoke LaSalle's invariance principle to conclude that all trajectories converge to the largest invariant set contained in $\{\tilde x=0\}$.
The proof is completed analyzing the first equation of \eqref{errequ2}.

The proof of P3 is established proving the converse implication, {\em i.e.}, that the trajectories of the ${\cal L}_1$--AC are bounded implies stability of the plant in closed--loop with the PI. In point P2 we
proved that if the trajectories of \eqref{errequ2} are bounded \eqref{lim} holds true. Now, the system in the third equation of \eqref{errequ2} is an LTI system whose input, {\em i.e.}, $\tilde \theta^\top x$
converges to zero
and whose output $\col(x,u)$ is bounded, for all initial conditions $\col(x(0),u(0))$, consequently the matrix $\cala_0$ is stable.\\
\end{proof}
%
\section{Some Further Remarks}
\lab{sec3}
%
\noindent {\bf R1} The property P1 in Proposition \ref{pro0} underscores that the stabilization mechanism of ${\cal L}_1$--AC is independent of the parameter adaptation, instead it is an elementary linear
systems principle. {{As shown in the proposition, the effect of the adaptation appears as a  perturbation term $k \tilde \theta^\top x$ to the implementable PI controller that, if trajectories are bounded,
asymptotically converges to zero. This explains why in ${\cal L}_1$--AC it is suggested to increase the adaptation gain---hoping that this term will die--out quickly. Moreover,  ${\cal L}_1$--AC includes a
parameter projection that, due to the use of utterly high adaptation gains, induces a bang--bang--like behavior in the estimate that, in average, behaves like a constant value. See \cite{BOSMEH} for some
conclusive simulated evidence.}}

\noindent {\bf R2} The qualifier ``implementable" is essential to appreciate the significance of our results. Of course, all (linearly parameterized) adaptive controllers can be implemented as an LTI system
perturbed by the parameter error but the resulting LTI system depends on {\em unknown plant} parameters. Due to the inclusion of the input filter, this is not the case in  ${\cal L}_1$--AC---rendering
irrelevant the use of adaptation. {  In \cite{ortpanscl} this deleterious effect of the input filter has been shown to be pervasive for all model reference controller structures, not just the state--feedback,
canonical system representation treated in this paper.

\noindent {\bf R3} In \cite{ortpanifac} it has been shown that there exists $k_c>0$ such that the  PI controller \eqref{pi1}, \eqref{pigai} ensures global asymptotic stability of the closed--loop system for
all $k>k_c$, all unknown parameters $a_i,i \in \bar n$, and all Hurwitz matrices $A_m$ of the form \eqref{am}. On the other hand, to the best of the authors' knowledge, it is not known whether there exists
suitable values of $\gamma$ and $k$ such that the origin of \eqref{errequ2} is (asymptotically) stable for all unknown parameters $a_i,i \in \bar n$, and all Hurwitz matrices $A_m$ of the form \eqref{am}.

\noindent {\bf R4} We have assumed for simplicity the case of regulation to zero and taken the input filter used in  ${\cal L}_1$--AC as $ D(s)={k \over s+k}. $ The proposition extends {\em verbatim} to the
case of nonconstant reference and general (stable, strictly proper) LTI filters $D(s)$. See  \cite{ortpanscl} and \cite{ortpanifac}.

\noindent {\bf R5} As shown in \cite{ortpanifac} the characteristic polynomial of $\cala_0$ satisfies\footnote{From \eqref{carpol1} it follows that $\cala_0$ may not have an eigenvalue at zero, but it may have
eigenvalues in the $j\omega$ axis. Hence, the stability statement in P3 of Proposition \ref{pro0} cannot be strengthened to asymptotic stability. The authors thank Denis Efimov for this insightful remark.} }
\begequ
\lab{carpol1}
\det (sI_n - \cala_0) =s\det (sI_n - A) + k\det (sI_n - A_m).
\endequ
From which it is clear that, if the plant is unstable, it is necessary to take ``large" values of $k$ to stabilize the ${\cal L}_1$--AC, see \eqref{errequ2}. This is in contradiction with the main promotional
argument of ${\cal L}_1$--AC, namely that ``it compensates for the mismatch between the ideal system and the plant within the frequency range of the lowpass filter $D(s)$". Moreover, it is recognized in
\cite{HOVetal} that ``the allowed bandwidth of the filter $D(s)$ is limited by robustness considerations"---contradicting, again, the need for large $k$.  It is interesting to note that in the limit, as $k \to
\infty$, from \eqref{sysu} we recover the good old model reference adaptive controller $u=\hat \theta^\top x$!

\noindent {\bf R6} In  \cite{HOVetal} it is argued that the inclusion of the input LTI filter and the use of large adaptation gains ``decouples the estimation and control loops"---a notion that is never
explained mathematically. In adaptive control ``decoupling" between the adaptation and the control loops is (partially) achieved using small adaptation gains that ensure the estimated parameters vary
slowly---with respect to the variation of the plant states. Leaving aside the numerical problems and unpredictable transient behavior generated with large adaptation gains, Point P2 of Proposition \ref{pro0}
clarifies this decoupling effect, namely, making the ${\cal L}_1$--AC converge faster to the implementable PI controller.

\noindent {\bf R7} In \cite{HOVetal} the condition\footnote{Actually, the condition (85) given in \cite{HOVetal} is far more restrictive than \eqref{consta}.}
\begequ
\lab{consta}
\|(pI_m - A_m)^{-1}b\theta^\top {p \over p+k}\|_\infty < 1
\endequ
where $p:={d \over dt}$ and $\|\cdot\|_\infty$ is the $\Linf$--induced operator norm,\footnote{For convolution operators $\|H\|_\infty=\|h(t)\|_1=\int_0^\infty |h(t)|dt$, where $h(t)$ is the impulse response
\cite{DESVID}.} is imposed to derive some $\Linf$ bounds on some suitably chosen signals---that, interestingly, do not include the tracking error. Clearly, this condition is far more restrictive than the
condition $k>k_c$ discussed in Remark R3 above, which ensures stability of the PI. As a matter of fact, in \cite{ortpan} it is shown that, for scalar systems, \eqref{consta} cannot be satisfied for all systems
and reference models.

\noindent {\bf R8}  The present paper extends the results of  \cite{ortpan}, where we treat only scalar systems. It is similar in spirit to the proof of \cite{HEUDUM} that output feedback ${\cal L}_1$--AC is,
actually, nonadaptive. It also complements the recent report \cite{NARetal} where the claims of robustness and performance improvement of ${\cal L}_1$--AC are scrutinized via theoretical analysis and a series
of numerical examples. The interested reader is also referred to \cite{BOSMEH,NARetal} where the issues of numerical instability due to high--gain adaptation and bang--bang behavior of the control due to
parameter projection ${\cal L}_1$--AC, are discussed. The inability of ${\cal L}_1$--AC to track non--constant references is widely acknowledged, see \cite{ortpan} for a particular example. A freezing property
of high--gain estimators, that puts a question mark on the interest of using it, is proven in \cite{BARORTAST}, see also \cite{ortpan}.

\noindent {\bf R9}  The motivation to crank up the adaptation gain in ${\cal L}_1$--AC is related with some transient performance bounds claimed by the authors. Indeed, it is easy to show  \cite{HOVetal} that,
if the initial conditions of the predictor {\em coincide with the initial conditions of the plant}, the prediction error is upper--bounded by a constant that is inversely proportional to the adaptation gain.
The intrinsic fragility of results relying on particular initial conditions is a key issue often overlooked by (mathematically oriented) researchers in our community. See Sidebar 2. 

\noindent {\bf R10}  It is not surprising that ${\cal L}_1$--AC has been successful in some applications. As shown above, it (essentially) coincides with a full state feedback PI controller that, as is
well--known, is robust and can reject  constant disturbances and track constant references, a scenario that seems to fit the realm of applications reported for this controller.
%
%

%
\section*{Sidebar 1: Comments Regarding \cite{KHAetal}}
%
In the abstract of  \cite{KHAetal}  one finds the following unusually candid
sentence: "the L1 adaptive controllers approximate an implementable
non-adaptive linear controller." In this sidebar we derive,
in a mathematically rigorous way, the calculations done in \cite{KHAetal} that motivated the previous sentence and place it in the context of our
work.

Towards this end, consider the LTI system
$$
y(s) = G(s)u(s),
$$
where $G(s) \in \rea(s)$ and is strictly proper. This corresponds to
equation\footnote{Numbers in brackets refer to the equations of   \cite{KHAetal}.}  \{1\} in  \cite{KHAetal}, where the symbol $A(s)$ is used instead of
$G(s)$, and a much more general plant model is treated. We consider
here the simplest scenario needed to convey our message. The plant
model is rewritten in \{2\} as 
\begequ
\lab{pla}
y(s) = M(s)[u(s) + \sigma(s)],
\endequ
where
\begequ
\lab{sig}
\sigma(s)=\left[G(s) \over M(s) -1\right]u(s), 
\endequ
with $M(s) \in \rea(s)$, stable of relative degree smaller than the
relative degree of $G(s)$. The ${\cal L}_1$--AC  \{7\} is given in this case by
\begequ
\lab{ulone}
u(s) = C(s)[r(s) - \hat \sigma(s)],
\endequ
where $r(t)$ is some reference signal, $C(s) \in \rea(s)$, is strictly proper,
stable and verifies $C(0) = 1$, and $\hat \sigma(t)$ is a signal that represents
an "estimate" of $\sigma(t)$ generated with the estimator \{4\}--\{6\}.

In \{8\} and \{9\} of \cite{KHAetal} the following signals, called reference
signals, are introduced
\begequarr
\nonumber
y_r(s) & = & M(s)[u_r(s) + \sigma_r(s)]\\
\nonumber
u_r(s) & = & C(s)[r(s) - \sigma_r(s)]\\
\lab{refsys}
\sigma_r(s)& =& {G(s) \over M(s)-1}u_r(s).
\endequarr
It important to note that \eqref{refsys} exactly coincides with \eqref{pla}--\eqref{ulone} when
$$
\hat \sigma(t) \equiv \sigma(t),
$$
which corresponds to the case {\em without adaptation and known plant parameters}.

It is then claimed, without proof, that the ${\cal L}_\infty$ norm of the
errors $y(t)-y_r(t)$ and $u(t)-u_r(t)$ can be made arbitrarily small "reducing the sampling time" of the estimator \{4\}--\{6\}. This leads
the authors to affirm that "the ${\cal L}_1$--AC system {\em approximates} the
reference system \eqref{refsys}".

The authors then proceed with some transfer function
manipulations to establish \{13\}, that is,
\begequ
\lab{uref}
u_r(s) = {C(s) \over [1 - C(s)]M(s)}[M(s)r(s) - y_r(s)].
\endequ
This is referred to as {\em limiting controller} and it is stated that "it
is equivalent to the ${\cal L}_1$--AC under {\em fast adaptation}". Notice that the
controller \eqref{uref}---for a plant with output $y_r$---can be implemented
without knowledge of the original plant parameters. This is the
justification given to the sentence in the abstract of   \cite{KHAetal} mentioned
above.

It is, furthermore, argued that the "architectures of both controller are fundamentally different". The motivation for this statement is
that, in contrast with ${\cal L}_1$--AC, the implementation of the limiting
controller involves the inversion of M(s) while "the estimation
loop in ${\cal L}_1$--AC computes the approximate desired systems inverse".

For the case treated in our paper
$$
C(s) = {k \over  s + k}, \quad \sigma(t)  =  \theta^\top x(t), \quad \hat \sigma(t)  =  \hat \theta^\top(t) x(t).
$$
For scalar plants the controller \eqref{uref}, with 
$$
y_r(t) \equiv y(t), \quad M(s)={1 \over s+a_m}, \quad r(t) \equiv 0,
$$
exactly coincides with the PI
controller reported in our paper.

As usual in the manipulation of stable transfer functions the calculations
of   \cite{KHAetal} neglect the exponentially decaying terms due to initial conditions. This kind of assumption is untenable in nonlinear systems, like adaptive controllers, since it is well--known that even for globally exponentially stable systems the trajectories can be driven to infinity when perturbed by exponentially decaying disturbances \cite{TEEHES}.

In Proposition \ref{pro0} we consider
general $n$--th order plants with {\em arbitrary} initial conditions, and prove
that the output of the  ${\cal L}_1$--AC---including the parameter estimator with {\em arbitrary} adaptation gain---exactly coincides with the output of an implementable PI perturbed
by a term coming from the adaptation. It is also established that
this term converges to zero, therefore the  ${\cal L}_1$--AC alway converges
to the PI. It is easy to prove that the latter result cannot be recovered
with the transfer function manipulations of   \cite{KHAetal}. Indeed, defining in the standard way
an estimation error
$$
\tilde \sigma(t) := \hat \sigma(t)-\sigma(t),
$$
and doing the calculations above with the original plant it follows that the
control signal \eqref{ulone} can be written as
$$
u(s) ={C(s) \over [1 - C(s)]M(s)} \left[M(s)r(s) - y(s)\right] +{C(s) \over 1 - C(s)} \tilde \sigma(s).
$$
Unfortunately, from this equation we cannot conclude that the  ${\cal L}_1$--AC
converges to the implementable LTI controller. Indeed, because of
the constraint $C(0) = 1$, the perturbing term $\tilde \sigma(t)$ passes through
an integrator. Since  we do not know if this signal is integrable we cannot even claim that its contribution to the
control signal is bounded---let alone converging to zero.

In the light of this discussion, the interest of the "approximation"
and "equivalence" statements of  \cite{KHAetal} is questionable and they
are certainly far from substantiating the claim made in the abstract
of the paper---a fact that Proposition \ref{pro0} presents in a crystal--clear
manner.
%
\section*{Sidebar 2: Comments Regarding the Analysis done in the  ${\cal L}_1$--AC Literature}
As indicated in Remark R9 trajectory--dependent claims in control theory are intrinsically fragile.  Indeed, the whole body of control theory has been
developed to design controllers whose performance is guaranteed {\em independently of the initial conditions}---for instance, ensure stability (in the sense of Lyapunov) of a desired equilibrium point or finite gain of an operator. When the result is valid for a specific initial condition, this property is valid only for that specific trajectory, and cannot be extrapolated to any other one. Hence, in the presence of unknown and unpredictable disturbances, or measurement errors, that will drive away the state from that trajectory,
nothing can be said about the new trajectory. 

To illustrate this point consider the simple case of the LTI plant 
$$
\dot x_1 = - x_1,\quad \dot x_2 = x_2-x_1.
$$
Clearly, for all initial conditions in the set $\{x \in \rea^2\;|\; x_1=x_2\}$, the corresponding trajectory is bounded and converges to zero. But, obviously, all trajectories starting outside this set grow unbounded. 

Unfortunately, the analysis of  ${\cal L}_1$--AC reported in the literature is trajectory dependent since it relies on the assumption that the initial state of the estimator \eqref{syshatx} {\em coincides} with the initial state of the plant, that is $\hat x(0)=x(0)$.  Since the plant state cannot be exactly measured, or maybe subject to disturbances that would require an unpractical resetting of the estimator, this condition is always violated in practice. 
 
\end{document}